\documentclass[final]{siamltex}
\usepackage[T1]{fontenc}

\usepackage{amssymb,latexsym,amsmath}
\usepackage{mathrsfs}
\usepackage{enumitem}
\usepackage{caption}
\usepackage{comment}
\usepackage{inputenc}
 \usepackage{makeidx}
\usepackage{color}

\allowdisplaybreaks

\newtheorem{example}{Example}[section]
\newtheorem{remark}{Remark}[section]

\allowdisplaybreaks

\usepackage{epsfig} 
\usepackage{verbatim}

\usepackage{amsfonts}
\usepackage{float}
\usepackage{multirow}
\title{Continuity Properties of Value Functions in Information Structures for Zero-Sum and General Games and Stochastic Teams \thanks{Supported
    by the Natural Sciences and Engineering Research Council of Canada. A preliminary version of this paper was presented at the 2021 IEEE Conference on Decision and Control.}}

\author{Ian Hogeboom-Burr \and Serdar Y\"{u}ksel\thanks{The authors are with the Dept. of Mathematics and Statistics, Queen's University, Kingston K7L 3N6, ON, Canada, {\tt\small \{15ijhb,yuksel\}@queensu.ca}.}}

\begin{document}
\maketitle

\begin{abstract}
We study continuity properties of stochastic game problems with respect to various topologies on information structures, defined as probability measures characterizing a game. We will establish continuity properties of the value function under total variation, setwise, and weak convergence of information structures. Our analysis reveals that the value function for a bounded game is continuous under total variation convergence of information structures in both zero-sum games and team problems. Continuity may fail to hold under setwise or weak convergence of information structures, however, the value function exhibits upper semicontinuity properties under weak and setwise convergence of information structures for team problems, and upper or lower semicontinuity properties hold for zero-sum games when such convergence is through a Blackwell-garbled sequence of information structures. If the individual channels are independent, fixed, and satisfy a total variation continuity condition, then the value functions are continuous under weak convergence of priors. We finally show that value functions for players may not be continuous even under total variation convergence of information structures in general non-zero-sum games. 
\end{abstract}

\section{Introduction}

In stochastic decision, game, and control problems, the probability measures and models on the exogenous state and measurement variables (on which decisions are measurable) may be called information structures. The setup will be made precise later in the paper.

For single-decision maker setups, comparison of information structures (also called statistical experiments / observation channels) was conclusively studied by Blackwell in his seminal paper \cite{Blackwell1951} for finite probability models. 
Since then, there has been significant interest in the comparison, value and regularity of information structures in various research communities and into more general models \cite{strassen1965existence,torgersen1991comparison,LeCamReview}. For finite games, \cite{pkeski2008comparison} generalized Blackwell's result to zero-sum games with a complete characterization. This result was generalized to the standard Borel zero-sum game setup in \cite{HBY2020arXiv}, where a partial converse of Blackwell's information structure ordering and relations with a well-known result due to Strassen \cite{strassen1965existence} were also established with a corollary being that more information does not hurt the agent receiving it. 


A related problem involves not only the comparison of two information structures but the continuity properties of optimal solutions/equilibrium solutions in information structures under various topologies, which is studied in this paper.

\section{Preliminaries, Literature Review and Statement of Main Results}

\subsection{Information Structure Models and Assumptions}

Consider $n \in \mathbb{Z}_{\geq 1}$ players in a single-stage game. Let $\mathbb{X}$ denote a standard Borel state space, with $x \sim \zeta$ an $\mathbb{X}$-valued random variable known as the state of nature. Recall that a standard Borel space is a Borel subset of a complete, separable, metric (Polish) space. Let $\zeta$ be the {\it prior} probability distribution on a hidden state, which is common knowledge to all players. We denote by $\mathbb{Y}^i$ Player $i$'s standard Borel measurement space, for $i \in \{1, \dots, n\}$. We use $-i$ to denote all players other than Player $i$. For each player, we define their measurement $y^i$, which is a $\mathbb{Y}^i$-valued random variable, where: $y^i = g^i(x, v^i)$ for some noise variable $v^i$ (and which, without any loss, can be taken to be $[0,1]$-valued). By stochastic realization arguments \cite{BorkarRealization}, the above formulation is equivalent to viewing $y^i$ as being generated by a measurement channel $Q^i$, which is a Markov transition kernel from $\mathbb{X}$ to $\mathbb{Y}^i$.
We denote the joint probability measure on $\mathbb{X} \times \mathbb{Y}^1 \times \dots \times \mathbb{Y}^n$ by $\mu(dx, dy^1, \dots, dy^n)$, and define this as the \textit{information structure}.

Throughout this paper, we allow for the entire information structure to vary as it converges, meaning the prior on the state space $\zeta$ is not fixed. The sole exceptions are Theorem \ref{zerosumweak} and Theorem \ref{zerosumsetwise}, where the prior must be fixed for the proofs to hold. 

For $\zeta\in  \mathcal{P}(\mathbb{X})$ and kernel $Q$, we let $\zeta Q$ denote the joint distribution induced on
$(\mathbb{X}\times \mathbb{Y}, \mathcal{B}(\mathbb{X}\times
\mathbb{Y}))$ by channel $Q$ with input distribution $\zeta$:
\[  \zeta Q(A) = \int_{A} Q(dy|x)\zeta (dx), \quad A\in  \mathcal{B}(\mathbb{X} \times \mathbb{Y}). \]

Each player also has a standard Borel action space $\mathbb{U}^i$, and a personal cost function $c^i(x,u^1,\dots,u^n): \mathbb{X} \times \mathbb{U}^1 \times \dots \times \mathbb{U}^n \rightarrow \mathbb{R}$. 

For fixed prior, state space, and measurement spaces $\zeta$, $\mathbb{X}$, $\mathbb{Y}^1, \dots, \mathbb{Y}^n$, a \textit{game} is a $2n$-tuple consisting of a measurable and bounded cost function and an action space for each player, $G = (c^1, \dots, c^n, \mathbb{U}^1, \dots, \mathbb{U}^n)$. 

The player's goal is to minimize their expected cost functional for a given game $G$:
\begin{align*}
J^i(G, \mu, \gamma^1, \dots, \gamma^n) := E^{\mu, \bar{\gamma}}[c^i(x,\gamma^1(y^1), \dots, \gamma^n(y^n))]
\end{align*}
where the players select their policies from the set of all admissible \textit{policies} $\Gamma^i := \{\gamma: \mathbb{Y}^i \rightarrow \mathbb{U}^i\}$, which are measurable functions from a player's measurement space to their action space. We refer to $u^i = \gamma^i(y^i)$ as the action of the player, and $\gamma^i$ as their policy.We now define the following assumptions, which will be used within the results:

\textbf{Assumptions.}

\textit{A1}: The cost functions for players are measurable and bounded. 

\textit{A2}: The cost functions for players are continuous and bounded. 

\textit{A3}: The information structure is absolutely continuous with respect to a product measure:
\[P(dy^1, dy^2,\cdots, dy^n, dx) \ll \bar{Q}^1(dy^1)\bar{Q}^2(dy^2)\cdots\bar{Q}^n(dy^n)\zeta(dx),\]
for reference probability measures $\bar{Q}^i$, $i=1,2,\cdots$. That is, e.g., for $n=2$, there exists an integrable $f$ which satisfies for every Borel $A, B, C$
\begin{align*}
&P(y^1 \in B, y^2 \in C, x \in A) = \int_{A,B,C} f(x,y^1,y^2) \zeta(dx)\bar{Q}^1(dy^1)\bar{Q}^2(dy^2). 
\end{align*}

\textit{A4}: The action space of each agent is compact.


\textit{A5}: The bounded measurable cost function is continuous in players' actions for every state of nature $x$, and the action space of each player is compact. 

\textit{A6}: The individual player channels $Q^i$ are fixed and independent (i.e., given $x$, all the measurement variables $y^i$ are conditionally independent). Furthermore, each $Q^i$ is continuous in total variation, i.e., as $x_m \rightarrow x$ then $\|Q^i(\cdot|x) - Q^i(\cdot|x_m)\|_{TV} \rightarrow 0$.

\textit{A7}: The action spaces for the players are convex subsets of $\mathbb{R}^k$, for some $k$.

\subsection{Convergence of information structures}\label{convergenceIS}

For a standard Borel space $\mathbb{X}$, let $\mathcal{P}(\mathbb{X})$ denote the family of
all probability measures on $(\mathbb{X},\mathcal{B}(\mathbb{X}))$, where $\mathcal{B}(\mathbb{X})$ denotes the Borel sigma-field over $\mathbb{X}$. 
Let $\{\mu_n,\, n\in \mathbb{N}\}$ be a sequence in $\mathcal{P}(\mathbb{X})$. Recall that
$\{\mu_n\}$ is said to  converge
  to $\mu\in \mathcal{P}(\mathbb{X})$ \emph{weakly} if
\begin{eqnarray}\label{Defnconv}
 \int_{\mathbb{X}} c(x) \mu_n(dx)  \to \int_{\mathbb{X}}c(x) \mu(dx)
\end{eqnarray}
for every continuous and bounded $c: \mathbb{X} \to \mathbb{R}$.
On the other hand, $\{\mu_n\}$ is said to  converge
  to $\mu\in \mathcal{P}(\mathbb{X})$  \emph{setwise} if (\ref{Defnconv}) holds
for every measurable and bounded $c: \mathbb{X} \to \mathbb{R}$. Setwise convergence  can also be defined through pointwise convergence on Borel subsets of $\mathbb{X}$, that is $\mu_n(A) \to \mu(A)$, for all $A \in{\cal B}(\mathbb{X})$.
For two probability measures $\mu,\nu \in
\mathcal{P}(\mathbb{X})$, the \emph{total variation} metric is given by
\begin{eqnarray}
\|\mu-\nu\|_{TV}:= 2 \sup_{B \in {\cal B}(\mathbb{X})} |\mu(B)-\nu(B)| =  \sup_{f: \, \|f\|_{\infty} \leq 1} \bigg| \int f(x)\mu(dx) -
\int f(x)\nu(dx) \bigg|, \label{TValternative}
\end{eqnarray}
A sequence  $\{\mu_n\}$ is said to  converge
  to $\mu\in \mathcal{P}(\mathbb{X})$ in total variation if
$\| \mu_n - \mu   \|_{TV}  \to 0.$ Setwise convergence is equivalent to pointwise convergence on
Borel sets  whereas total variation requires
uniform convergence on  Borel sets. Thus  these  three convergence
notions are in increasing order of strength. 

We also recall the $w$-$s$ topology introduced by Sch\"al \cite{schal1975dynamic}: the $w$-$s$ topology on the set of probability measures $\cal{P}(\mathbb{X} \times \mathbb{Y})$
is the coarsest topology under which, for every measurable and bounded $f(x,y)$ which is continuous in $y$
for every $x$, the map $\int f(x,y) \mu(dx,dy): \cal{P}(\mathbb{X} \times \mathbb{Y}) \to \mathbb{R}$
is continuous (but unlike the weak convergence topology, $f$ does not need to be continuous in $x$). An important relevant result \cite[Theorem 3.10]{schal1975dynamic} or \cite[Theorem 2.5]{balder2001} is that if $\mu_n \to \mu$ in weakly but if the marginals $\mu_n(dx \times \mathbb{Y}) \to \mu(dx \times \mathbb{Y})$ setwise, then the convergence is also in the $w$-$s$ sense. 

\begin{definition}
A sequence of information structures $\mu_n$ in ${\cal P}(\mathbb{X} \times \mathbb{Y}^1 \times \dots \times \mathbb{Y}^n)$ converges to $\mu$ weakly/setwise/in total variation if $\mu_n \to \mu$ in the corresponding sense.
\end{definition}

\subsection{Literature review}

\subsubsection{Single Decision-Maker/Player Problems}
For a single decision maker setup, \cite[Theorem 3.4]{YukselOptimizationofChannels}, for a fixed prior, showed that the value function is continuous under total variation convergence of information structures. Counterexamples (see \cite[Sections 3.1.1. and 3.2.1]{YukselOptimizationofChannels}) also revealed that continuity is not necessarily preserved under weak or setwise convergence, but the value function is upper semicontinuous under both under mild conditions. 

The case where priors also change has been studied in \cite{KYSICONPrior}, where conditions for optimality under both total variation and weak convergence were established: total variation convergence of the priors always leads to convergence \cite[Theorem 2.9]{KYSICONPrior}, but continuity under the weak convergence of priors requires a total variation continuity on the information channels \cite[Theorem 2.5]{KYSICONPrior}.

In single player setups, a related result is due to Le Cam \cite{LeCamSufficiency}, to be discussed further below in some detail in the context of zero-sum games. We also note that a related result due to Wu and Verd\'u \cite{WuVerdu} establishes continuity of the quadratic (minimum mean-square estimation) error under weak convergence of the priors when the channel is additive, the additive noise has a finite variance and it admits a continuous and bounded density function (such as a Gaussian). \cite{AbayaWise}, \cite{GrayDavisson} study the effects of uncertainties in the input distribution in the quantizer design, which may be viewed as a decision problem. \cite{dupuis2000kernel,UgrPet2001} study the effects of variations in system models, and thus also information structures, in a general relative entropy perturbation framework in the context of minimax LQG control; we note that relative entropy is a more stringent distance measure via Pinsker's inequality \cite{csiszar1967information}.

In this context, for completeness, it is may also be useful to note some information theoretic studies: In addition to the relationships between information structures via Blackwell's comparison of experiments and Le Cam's work on deficiencies, there is a comparison criterion for channels (and the joint map from a source and channel output after coding and decoding) developed by Shannon in \cite{ShannonComparison} and expanded upon in \cite{raginsky2011shannon}. This leads to a comparison criterion for information structures in single-player decision problems, based on communication-theoretic relaxations of such problems.
 


\subsubsection{Zero-Sum Games}

The specific form of an information structure has been shown to have subtle impact on (different types of) equilibria in games, as well as on their existence, uniqueness, and characterization (see for example \cite{witsenhausen1971relations, basols99, basar1989time}). In spite of the existence of several studies on the impact of information structure on equilibria of games, there does not exist a complete theory of such impact. The theory of and regularity properties on information structures in stochastic games is significantly more challenging when compared with that for stochastic team theory, since the value of information is not necessarily positive and informational changes lead to variations in the equilibrium behaviour in intricate ways \cite{Bassan2003,kamien1990value,tBasarStochasticDiffGames}). We also recall here some comprehensive texts studying informational aspects of stochastic game theory \cite{mertens2015repeated,osborne1994course,fudenberg1991game,basols99} which also study various informational regularity and existence properties of equilibrium solutions for stochastic games as well as continuity and comparison results for zero-sum games \cite[Chapter 3]{mertens2015repeated}.

Zero-sum games, however, form a class of problems where strong regularity properties can be established. For example, under mild conditions as reviewed below, every zero-sum game, and accordingly every information structure, has a value and as noted earlier a theory of ordering of information structures can be obtained and, in particular, no additional information can hurt a player, though this argument is quite subtle in the game case when compared with the team setup as {\it choose to ignore the additional information and thus no information cannot hurt} argument (typically used in control theory or team theory) is not applicable; though the implication that more information cannot hurt turns out to be correct via a much more tedious (geometric) argument; see \cite{pkeski2008comparison}  \cite{HBY2020arXiv}. 

In view of the above, we can somehow interpret zero-sum games as an intermediate category between team problems and non-zero sum game problems as far as their informational properties are concerned. The findings of our paper will further contribute towards such an interpretation.

In the zero-sum game case there are two players who share a common cost function. Player 1's goal is to minimize the expected cost, while Player 2's goal is to maximize it. The players are commonly referred to as the minimizer and the maximizer, respectively.  Given fixed $\mathbb{X}$, $\mathbb{Y}^1$, $\mathbb{Y}^2$, and $\zeta$ such that $x \sim \zeta$, a zero-sum game $g = (c, \mathbb{U}^1, \mathbb{U}^2)$ is a triple of a measurable cost function $c: \mathbb{X} \times \mathbb{U}^1 \times \mathbb{U}^2 \rightarrow \mathbb{R}$ and standard Borel action spaces for each player $\mathbb{U}^1, \mathbb{U}^2$. 

%

\begin{definition} Given an information structure $\mu$, we say that $\gamma^{1,*}, \gamma^{2,*}$ is an equilibrium for a zero-sum game $g$ if
\begin{eqnarray*}\inf_{\gamma^1 \in \Gamma^1} J(g, \mu,\gamma^{1},\gamma^{2,*}) = J(g, \mu,\gamma^{1,*},\gamma^{2,*}) = \sup_{\gamma^2 \in \Gamma^2}  J(g, \mu,\gamma^{1,*},\gamma^2). \nonumber\end{eqnarray*}
\end{definition}

Let $J^{*}(g, \mu)$ be $J(g, \mu, \gamma^{1}, \gamma^{2})$ where $(\gamma^{1}, \gamma^{2})$ are chosen to be the equilibrium strategies for the players. In the standard Borel setup, sufficient conditions can be found for the existence of Nash-equilibrium in zero-sum games.  \cite[Theorem 3.1]{balder1988generalized} (see also \cite[Theorem 3.1]{HBY2020arXiv}) presents an existence result under an absolute continuity condition on information structures with respect to a reference measure. We also refer the reader to \cite{mamer1986zero} which presents complementary conditions where the absolute continuity condition is relaxed. Finally, \cite[Theorem 3.4]{balder1988generalized} presents a generalization where action sets are information dependent. The essence of these results were presented in \cite[Theorem 3.2]{HBY2020arXiv}.

Among contributions, \cite{EinyDifferential} shows uniform continuity of the value function for zero-sum games under convergence of player information fields using a notion of convergence of sigma fields developed by Boylan in \cite{boylan}, under Lipschitz continuity assumptions on the cost function; while \cite{EinyContinuity} establishes continuity of the value function in the prior distribution on the state space for zero-sum games using the total variation metric. In \cite{pkeski2019value}, for countable $\mathbb{X}, \mathbb{Y}^1, \mathbb{Y}^2$, a \textit{value-distance} is introduced, which is a method of comparison of information structures that can be characterized using the total variation distance in zero-sum games. 

\begin{definition}[\cite{pkeski2019value}]\label{valuedistance}
Take countable $\mathbb{X}, \mathbb{Y}^1, \mathbb{Y}^2$. Let $\tilde{\mathbb{G}}$ be all games with cost functions satisfying A1 (and where the cost is bounded by $1$) and countable action spaces $\mathbb{U}^1$ and $\mathbb{U}^2$. The value-distance between two information structures for zero-sum games is: $d_2(\mu,\nu) =: \sup_{g \in \tilde{\mathbb{G}}} |J^*(g, \mu) - J^*(g,\nu)|$.
\end{definition}

\begin{definition}\label{garbling2}
We denote by $\kappa^1 \mu$ the information structure in which Player $1$'s information from $\mu$ is \textit{garbled} by a stochastic kernel $\kappa^1$, and $\mu \kappa^2$ the information structure in which Player 2's information from $\mu$ is garbled by $\kappa^2$. Explicitly, this means the information structure becomes:
\begin{align}
    (\kappa^{1}\mu)(B,dy^{2}, dx)  &= \int_{\mathbb{Y}^1}\kappa^{1}(B|y^1)\mu(dy^1,dy^{2},dx), \: B \in \mathcal{B}(\mathbb{Y}^1), \nonumber \\
     (\mu\kappa^2)(dy^1,B, dx)  &= \int_{\mathbb{Y}^2}\kappa^{2}(B|y^2) \mu(dy^1,dy^{2},dx), \: B \in \mathcal{B}(\mathbb{Y}^2). \nonumber
\end{align}
\end{definition}
We use $K^i$ to denote the space of all such stochastic kernels $\kappa^i$ for player $i$.

\begin{theorem}[Theorem 1, \cite{pkeski2019value}]\label{peskioriginal}
Take countable $\mathbb{X}, \mathbb{Y}^1, \mathbb{Y}^2$. The following equality holds:
\begin{align*}
   d_2(\mu, \nu) &= \max\{\min_{\kappa^1 \in K^1,\kappa^2\in K^2}\|\kappa^1\mu - \nu\kappa^2\|_{TV}, \min_{\kappa^1 \in K^1,\kappa^2 \in K^2}\|\mu\kappa^2 - \kappa^1\nu\|_{TV}\}. \nonumber
\end{align*}
\end{theorem}
Via this distance, it follows that convergence of $\{\mu_n\}$ to $\mu$ in total variation results in convergence of $J^*(g, \mu_n)$ to $J^*(g,\mu)$ when the state, measurement and action spaces are countable. 
In our paper we will consider general standard Borel spaces and demonstrate that if $\{\mu_n\}$ converges to $\mu$ in total variation, then $J^*(g, \mu_n)$ converges to $J^*(g,\mu)$ for all games $g$ for which Nash-equilbria exist under $\{\mu_n\}$ and $\mu$. Furthermore, we will establish upper semi-continuity properties that $J^*(g, \cdot)$ exhibits under weak and setwise convergence of measurement channels for each player. 

We also note that Theorem \ref{peskioriginal} is closely related in the single-player case to a result from Le Cam in \cite{LeCamSufficiency}. We recall the definition of the Le Cam distance for a single-player decision problem with information structures $\mu$ and $\nu$.

\begin{definition}
The Le Cam distance between two information structures $\mu$ and $\nu$ is
\begin{align}
    \Delta(\mu, \nu) &:= \max\{\delta(\mu,\nu), \delta(\nu,\mu)\} \nonumber \\ &:= \max\{\inf_{\kappa \in K}\sup_{x \in \mathbb{X}}\|\kappa \mu (\cdot|x) - \nu(\cdot|x)\|_{TV}, \inf_{\kappa \in K}\sup_{x \in \mathbb{X}}\|\mu(\cdot|x) - \kappa \nu(\cdot|x)\|_{TV}\}. \nonumber
\end{align}
\end{definition}
$\delta(\mu,\nu)$ is referred to as the \textit{Le Cam deficiency of $\mu$ with respect to $\nu$}. We now recall the following theorem adapted from Le Cam \cite{LeCamSufficiency}, which holds in the standard Borel setup. 

\begin{theorem}[Theorem 3, \cite{LeCamSufficiency}]\label{LeCam}
For a given $\epsilon > 0$, if $\delta(\mu, \nu) \leq \epsilon$ then for any policy $\gamma_{\nu}$ under $\nu$ and any single-player game with a bounded cost function $\|c\|_{\infty} \leq 1$, there exists a policy $\bar{\gamma}_{\mu}$ under $\mu$ such that
\begin{equation}
     E^{\mu}_{\zeta}[c(x, \bar{\gamma}_{\mu}(y))]\leq E^{\nu}_{\zeta}[c(x, \gamma_{\nu}(y))] + \epsilon. \nonumber
\end{equation}
\end{theorem}

\begin{remark} We observe that the ``value-distance'' in Definition \ref{valuedistance} is similar to the Le Cam distance, with the difference being that the Le Cam distance looks at maximizing the total variation distance over individual states, rather than incorporating the prior. Theorem \ref{LeCam} is also closely related to Theorem \ref{peskioriginal} in the single-player case. Theorem \ref{LeCam} implies that if $\delta(\mu,\nu) = \epsilon$ then the maximum possible decrease in value over all valid games when changing from $\mu$ to $\nu$ is upper-bounded by $\epsilon$. Similarly, if $\delta(\nu,\mu) = \alpha$ then the maximum possible increase in value going from $\mu$ to $\nu$ is upper-bounded by $\alpha$. 
If we define $d_1$ similar to Definition \ref{valuedistance}, as the supremum of the absolute difference in value over all valid single-player games, as a natural restriction of $d_2$ to the single-player case, we would have that $d_1(\mu, \nu) \leq \max\{\delta(\mu,\mu),\delta(\nu,\mu)\}$, and so $d_1(\mu,\nu) \leq \Delta(\mu,\nu)$. Thus, the Le Cam distance gives an upper-bound on the value-distance. 
\end{remark}


\subsubsection{Stochastic Team Problems}
The induced cost under a collection of policies $\gamma^1,\cdots, \gamma^n$ for a team problem is given by
\[J(c, \mu, \bar{\gamma}) =: \int c(x, \gamma^1(y^1), \dots, \gamma^n(y^n))\mu(dx, dy^1, \dots, dy^n)\]

A game for a team-problem is an $n+1$-tuple consisting of a common cost function $c: \mathbb{X} \times \mathbb{U}^1 \times \dots \times \mathbb{U}^n \rightarrow \mathbb{R}$ and standard Borel action spaces $\mathbb{U}^i$ for each player $i \in \{1,\dots,n\}$. 

In team problems, equilibrium solutions are team policies $\bar{\gamma}^* := (\gamma^{1,*}, \dots, \gamma^{n,*})$ which minimize the value function. The value function at equilibrium is denoted by $J^*(c, \mu)$. General existence results for team-optimal policies can be found in \cite[Section 5]{YukselWitsenStandardArXiv}. 


\subsubsection{General Stochastic Games} The value of information may not be well-posed since unlike team and zero-sum game problems, one cannot talk about the unique value of a general non-zero-sum game. Furthermore, more information may have a positive or negative value to a player who receives it, unlike in zero-sum games and team problems. An example of such an occurrence can be found in \cite{Bassan2003}. The electronic mail game from \cite{ElectronicMail} is a further relevant classic example where equilibria do not converge (although the information does not converge in total variation in this example). We note that a somewhat different notion of $\epsilon$-equilibria (a uniform-$\epsilon$ equilibrium concept where $\epsilon$-proximity, with $\epsilon$ being uniform for each conditional expected cost/reward given the private information realization, regardless of how unlikely the realizations are) has been studied in \cite{kajii1998payoff} and \cite{rothschild2005payoff} for finite games, where it has been shown that total variation continuity does not hold under this concept.

In this paper, we will answer the question of whether convergence of information structures leads to convergence of the value function (assuming it is well-defined) for a player in a general non-zero-sum game. 

\subsection{Contributions of the paper}
In this paper, we make the following contributions. 
\begin{itemize}
\item[(i)] We show that the value function for a zero-sum game is continuous in total variation convergence of the information structure if the game has a bounded measurable cost function. In addition, when the prior is fixed, the value function is either upper or lower semicontinuous in setwise convergence of an information structure, if the sequence of information structures is a minimizer-garbling or maximizer-garbling sequence. The same results hold for weak convergence of the information structures, when the assumptions on the cost function are such that it is continuous and bounded, and the player action spaces are convex. If the channels are fixed and continuous in total variation, continuity under weak convergence of the prior holds. 
\item[(ii)] We show that the value function for team problems is continuous in total variation convergence of the information structure for bounded cost functions, upper semicontinuous in setwise convergence for measurable and bounded cost functions, and upper semicontinuous in weak convergence for games with continuous and bounded cost functions and convex action spaces. If the channels are fixed and continuous in total variation, continuity under weak convergence of the prior holds. 
\item[(iii)] A counterexample reveals that players in general non-zero-sum games may not have value functions that are continuous in total variation convergence of information structures, even for bounded cost functions. 
\end{itemize}

\section{Zero-Sum Games}\label{ZeroSumGames}

We have the following theorem.

%

\begin{theorem}\label{continuity}
\begin{itemize}
\item[(i)] For a fixed zero-sum game which satisfies Assumptions A1, A3, A4 and A5, the value function is continuous in information structures under total variation convergence of information structures, in the sense that if $\|\mu_n - \mu\|_{TV} \rightarrow 0$, then $|J^*(g,\mu_n) - J^*(g,\mu)| \rightarrow 0$ for all $\mu_n$ satisfying A3. 

\item[(ii)] The value function is not necessarily continuous under weak or setwise convergence of information structures even under Assumptions A2 and A4 and with a fixed prior. 
\end{itemize}
\end{theorem}

Before we state the proof, if existence of equilibria is assumed apriori, then A3 and A5 can be relaxed for part (i).

\begin{proof}
(i)
By \cite{HBY2020arXiv}, A3, A4 and A5 guarantee that an equilibrium exists under $\mu$ and $\nu$. Assume that $\mu$ and $\nu$ are such that $\|\mu - \nu\|_{TV} \leq \epsilon$ for some $\epsilon >0$. Note that if $\mu_n$ satisfies A3 and $\mu_n \to \mu$ in total variation, then $\mu$ also has to satisfy A3 (to see this, let $\mu_n \ll \kappa$ for each $n$ implying that $\kappa(B) = 0 \implies \mu_n(B)=0$. If there were to exist $B$ with $\mu(B) >0 $ but with $\kappa(B)=0$, then this would imply that $\mu_n(B)=0$ for all $n$ but not for the limit measure $\mu(B)$. This would violate total variation convergence).

Without loss of generality, assume that $J^*(g,\mu) - J^*(g, \nu) \geq 0$ for some game $g$ (a symmetric argument can be applied in the case where $J^*(g,\mu) - J^*(g, \nu) \leq 0$).
Then we have the following:
\begin{align*}
&J^*(g,\mu) - J^*(g, \nu) \\
& = \int c(x, \gamma^{1,*}_{\mu}(y^1), \gamma^{2,*}_{\mu}(y^2))\mu(dx, dy^1, dy^2) - \int c(x, {\gamma}^{1,*}_{\nu}(y^1), {\gamma}^{2,*}_{\nu}(y^2))\nu(dx, dy^1, dy^2)  \\
&\leq \int c(x, {\gamma}^{1,*}_{\nu}(y^1), \gamma^{2,*}_{\mu}(y^2))\mu(dx, dy^1, dy^2) - \int c(x, {\gamma}^{1,*}_{\nu}(y^1), {\gamma}^{2,*}_{\mu}(y^2))\nu(dx, dy^1, dy^2).
\end{align*}
The inequality comes from perturbing the minimizer's equilibrium strategy in the first term, making the expected cost larger, and perturbing the maximizer's equilibrium strategy in the second term, making the expected cost smaller. Then, viewing ${\gamma}^{1,*}_{\nu}$ and $\gamma^{2,*}_{\mu}$ as fixed, we can view $c(\cdot, {\gamma}^{1,*}_{\nu}(\cdot), \gamma^{2,*}_{\mu}(\cdot))$ as a measurable and bounded function from $\mathbb{X} \times \mathbb{Y}^1 \times \mathbb{Y}^2 \rightarrow \mathbb{R}$. Let $M \in \mathbb{R}_{\geq 0}$ be such that $\|c\|_\infty \leq M$. We have:
\begin{align*}
    &\int c(x, {\gamma}^{1,*}_{\nu}(y^1), \gamma^{2,*}_{\mu}(y^2))\mu(dx, dy^1, dy^2) - \int c(x, {\gamma}^{1,*}_{\nu}(y^1), {\gamma}^{2,*}_{\mu}(y^2))\nu(dx, dy^1, dy^2) \\
    &\leq M\|\mu - \nu\|_{TV} \leq M\epsilon.
\end{align*}
Thus, if $\|\mu_n - \mu\|_{TV} \rightarrow 0$, then $|J^*(g,\mu_n) - J^*(g,\mu)| \rightarrow 0$. 

(ii) 
For stochastic control problems (which can be viewed as single-player games), \cite[Sections 3.1.1. and 3.2.1]{YukselOptimizationofChannels} establish that value functions are not necessarily continuous under weak or setwise convergence. These counterexamples directly extend to the zero-sum game case: one can extend such a single-player game to a two-player zero-sum game where the action of one player has no impact on the cost: e.g. $|\mathbb{U}^2|=1$ so that Player 2 has no freedom to select a control action and the information structure of Player 1 changes. 
\end{proof}\\

We now study weak and setwise convergences of information structures. 


\begin{definition}
For a fixed prior $\zeta$, a sequence of information structures $\{\mu\}_{k = 1}^{\infty}$ is a \textit{maximizer-garbling} (or a minimizer-degarbling) sequence if for every $j \in \{1, \dots, \infty\}$ there exist kernels $\kappa^1_j$ and $\kappa^2_j$ such that $\mu_{j + 1}\kappa^2_j = \kappa^1_j \mu_j $.
\end{definition}

\begin{definition}
For a fixed prior $\zeta$, a sequence of information structures $\{\mu\}_{k = 1}^{\infty}$ is a \textit{minimizer-garbling} (or a maximizer-degarbling) sequence if for every $j \in \{1, \dots, \infty\}$ there exist kernels $\kappa^1_j$ and $\kappa^2_j$ such that $\kappa^1_j\mu_{j + 1} = \mu_j \kappa^2_j $.
\end{definition}

We note that, following the results in \cite{pkeski2008comparison} and \cite{HBY2020arXiv}, the value function for any zero-sum game will be monotonically decreasing for a maximizer-garbling sequence and monotonically increasing for a minimizer-garbling sequence. 

\begin{example}[Maximizer-Garbling Sequences]
Consider two-players in a zero-sum game with state space $\mathbb{X} = [0,1]$ endowed with the uniform distribution as $\zeta$. Let $\mathbb{Y}^1 = \mathbb{Y}^2 = \mathbb{R}$. 
\begin{itemize}
\item[(i)] For $m \in \mathbb{Z}_{\geq 0}$, define $\mu_m$ by the following independent measurement channels for the players: Player 1 receives measurement $y^1$, which is drawn randomly from a Gaussian distribution with mean $x$ and variance $1 + \frac{1}{m}$; Player 2 receives measurement $y^2$ which is drawn randomly from a Gaussian distribution with mean $x$ and variance $1 - \frac{1}{m+1}$. Since Gaussian distributions with higher variances are garblings of Gaussian distributions with lower variances, this sequence is a maximizer-garbling sequence, since for any $i \in \mathbb{Z}_{\geq 0}$, the maximizer's channel under $\mu_{i+1}$ is a garbling of his channel under $\mu_i$, while the minimizer's channel under $\mu_{i}$ is a garbling of her channel under $\mu_{i+1}$. Furthermore, this sequence converges weakly to the information structure in which both players' channels have variance $1$. One can also construct such an example where we have $y^1=x+v^1_m$, and if the variance of $v^1_m$ decreases to zero, we will have weak convergence to a point distribution. 
\item[(ii)] Another example involves noiseless but quantized measurement channels: $y^1=Q_m^1(x)$, where $Q_m^1$ is a uniform quantization of $[0,1]$ into $2^m$ bins. In this case, the quantizers are garblings of one another (as they are successive refinements as $m$ increases) and the weak limit is the fully informative channel: $Q(dy|x)= \delta_{x}(dy)$.
\end{itemize}
\end{example}

We emphasize that not all information structures can be related to each other as either minimizer or maximizer-garbling sequences. 


\begin{theorem}\label{zerosumweak}
Let Assumptions A2, A3, A4 and A7 hold and the prior be fixed. Let $\mu_m$ be a sequence of information structures converging weakly to information structure $\mu$. If the sequence is a maximizer-garbling, then the value function is lower semicontinuous (in the sense that if $\mu_n \to \mu$ weakly and $\mu_n, \mu$ satisfy Assumption A3, then $\liminf_{n \to \infty} J^*(g,\mu_n) \geq J^*(g,\mu)$). If the sequence is a minimizer-garbling, then the value function is upper semicontinuous. 
\end{theorem}
\begin{proof}
Consider a minimizer-garbling sequence.  By A2, our cost function is bounded; let $M \in \mathbb{R}_{\geq 0}$ such that $\|c\|_{\infty} = M$. Let $\mu_{{\mathbb{Y}}^j}$ be the marginal of $\mu$ on its $(j + 1)$th component. Let $\gamma^j$ be an arbitrary policy for player $j$. We note that every Polish space is second countable (i.e. has a countable basis), since it is separable and metrizable. Thus, every standard Borel space is second countable. Then, by Lusin's theorem, using the fact that $\mathbb{U}^j$ is convex by assumption, we have that for any $\epsilon >0$, there exists a continuous function $f^j: \mathbb{Y}^j \rightarrow \mathbb{U}^j$ such that $\mu_{{\mathbb{Y}}^j}(\{y^j : f^j(y^j) \neq \gamma^j(y^j)\}) < \epsilon$.

Letting $B^2 := \{y^2 : f^2(y^2) \neq \gamma^2(y^2)\}$, we have:
\begin{align}
&\int_{\mathbb{X} \times \mathbb{Y}^1 \times \mathbb{Y}^2} c(x, \gamma^1(y^1), \gamma^{2}(y^{2}))\mu(dx, dy^1, dy^2) - M\epsilon \nonumber \\
& < \int_{\mathbb{X} \times \mathbb{Y}^1 \times \mathbb{Y}^2} c(x, \gamma^1(y^1), f^{2}(y^{2}))\mu(dx, dy^1, dy^2) \nonumber \\ 
& < \int_{\mathbb{X} \times \mathbb{Y}^1 \times \mathbb{Y}^2} c(x, \gamma^1(y^1), \gamma^{2}(y^{2}))\mu(dx, dy^1, dy^2) + M\epsilon  \label{zerosumineq}
\end{align}

This holds since the cost function is bounded by $[-M,M]$ and so the expected cost on $\mathbb{X} \times B^i \times \mathbb{Y}^{-i}$ is bounded by $[-M\epsilon, M\epsilon]$.
Denote by $\Gamma^j_C$ the policy space for player $j$ in which the player's policy is continuous from $\mathbb{Y}^j$ to $\mathbb{U}^j$. We note that if Player $1$'s strategy is fixed, Player $2$ receives the same cost if supremizing over policies in $\Gamma^2$ or $\Gamma^2_C$: that the value must be greater when $\Gamma^2$ is used is a consequence of $\Gamma^2_C \subset \Gamma^2$. That the value must be greater than or equal to that under $\Gamma^2$ when $\Gamma^2_C$ is used is a consequence of (\ref{zerosumineq}), since the policies, as well as $\epsilon$, were arbitrary. The same reasoning applies in the reverse case.

We now observe that, for any $m \in \mathbb{Z}_{>0}$:
\begin{align}
& \inf_{\gamma^1 \in \Gamma^1_C}\sup_{\gamma^2 \in \Gamma^2}\int c(x, \gamma^1(y^1), \gamma^{2}(y^{2}))\mu_m (dx,dy^1,dy^{2}) \nonumber \\
& = \inf_{\gamma^1 \in \Gamma^1}\sup_{\gamma^2 \in \Gamma^2}\int c(x, \gamma^1(y^1), \gamma^{2}(y^{2}))\mu_m (dx,dy^1,dy^{2}). \label{EquilEq}
\end{align}

To prove this, let $\gamma^{1,*}$ and $\gamma^{2,*}$ be the equilibrium policies under $\mu_m$. Then, following the discussion on Lusin's theorem, for any $\epsilon > 0$, there exists a continuous policy $\hat{\gamma}^{1} \in \Gamma^1_C$, which is equal to $\gamma^{1,*}$ except on a set $B^1$ such that $\mu_{m,{\mathbb{Y}^1}}(B^1) = \epsilon$. If Player 1 plays this strategy, then for any strategy that Player 2 selects in $\Gamma^2$, the value for the game will increase by at most $2M\epsilon$. This is because, if Player 2's arbitrary new strategy $\hat{\gamma}^2$ results in a change in performance on $A :=  \{\mathbb{X} \times B^1 \times \mathbb{Y}^2\}$, the maximum possible difference in expected cost on this set is $2M\epsilon$. If Player 2's new strategy results in better performance on $C := \{\mathbb{X} \times (\mathbb{Y}^1 \setminus B^1) \times \mathbb{Y}^2\}$, then the difference between playing $\hat{\gamma}^2$ and $\gamma^{2,*}$ is at most $2M\epsilon$, because $\hat{\gamma}^1$ is identical to $\gamma^{1,*}$ on this set, so a larger difference would necessarily contradict the fact that $\gamma^{2,*}$ is Player 2's best response to $\gamma^{1,*}$ under $\mu$ (since Player 2 could lose at most $2M\epsilon$ on set $A$ while employing this strategy against $\gamma^{1,*}$ instead of playing $\gamma^{2,*}$, and so by gaining more than $2M\epsilon$ on $C$, Player 2 would guarantee better expected performance playing $\hat{\gamma}^2$ against $\gamma^{1,*}$ over playing $\gamma^{2,*}$, violating the Nash-equilibrium condition). Thus, since Player 1 has a continuous strategy that guarantees a worst possible outcome of a $2M\epsilon$ increase in expected cost, we know that:
\begin{align*}
  & \inf_{\gamma^1 \in \Gamma^1_C}\sup_{\gamma^2 \in \Gamma^2}\int c(x, \gamma^1(y^1), \gamma^{2}(y^{2}))\mu_m (dx,dy^1,dy^{2}) \nonumber  \\
   & \leq \inf_{\gamma^1 \in \Gamma^1}\sup_{\gamma^2 \in \Gamma^2}\int c(x, \gamma^1(y^1), \gamma^{2}(y^{2}))\mu_m (dx,dy^1,dy^{2}) - 2M\epsilon \nonumber
\end{align*}

We also know that since $\Gamma^1_C \subset \Gamma^1$, we have:
\begin{align*}
  &    \inf_{\gamma^1 \in \Gamma^1_C}\sup_{\gamma^2 \in \Gamma^2}\int c(x, \gamma^1(y^1), \gamma^{2}(y^{2}))\mu_m (dx,dy^1,dy^{2})  \nonumber  \\
&   \geq \inf_{\gamma^1 \in \Gamma^1}\sup_{\gamma^2 \in \Gamma^2}\int c(x, \gamma^1(y^1), \gamma^{2}(y^{2}))\mu_m (dx,dy^1,dy^{2}) \nonumber.
\end{align*}

Since $\epsilon$ was arbitrary, (\ref{EquilEq}) follows. Applying a similar argument again, it follows that:
\begin{align}\label{ContinuousPolicies}
 & \inf_{\gamma^1 \in \Gamma^1_C}\sup_{\gamma^2 \in \Gamma^2_C}\int c(x, \gamma^1(y^1), \gamma^{2}(y^{2}))\mu_m (dx,dy^1,dy^{2}) \nonumber \\
& = \inf_{\gamma^1 \in \Gamma^1}\sup_{\gamma^2 \in \Gamma^2}\int c(x, \gamma^1(y^1), \gamma^{2}(y^{2}))\mu_m (dx,dy^1,dy^{2}).   
\end{align}

Now take arbitrary large $M \in \mathbb{Z}_{\geq 0}$. Then, we have:
 \begin{align}
& \inf_{\gamma^1 \in \Gamma^1} \sup_{\gamma^2 \in \Gamma^2}\int c(x, \gamma^1(y^1), \gamma^{2}(y^{2}))\mu_M(dx,dy^1,dy^{2}) \nonumber \\
 & = \inf_{\gamma^1 \in \Gamma^1_C} \sup_{\gamma^2 \in \Gamma^2_C}\int c(x, \gamma^1(y^1), \gamma^{2}(y^{2}))\mu_M(dx,dy^1,dy^{2}) \nonumber \\
 & \leq \inf_{\gamma^1 \in \Gamma^1_C}\sup_{\gamma^2 \in \Gamma^2_C} \limsup_{M \rightarrow \infty} \int c(x, \gamma^1(y^1), \gamma^{2}(y^{2}))\mu_M(dx,dy^1,dy^{2}) \label{inequality} \\
 & = \inf_{\gamma^1 \in \Gamma^1_C} \sup_{\gamma^2 \in \Gamma^2_C}\int c(x, \gamma^1(y^1), \gamma^{2}(y^{2}))\mu(dx,dy^1,dy^{2}) \nonumber \\
 & = \inf_{\gamma^1 \in \Gamma^1} \sup_{\gamma^2 \in \Gamma^2}\int c(x, \gamma^1(y^1), \gamma^{2}(y^{2}))\mu(dx,dy^1,dy^{2}) \nonumber
\end{align}

The first and final equalities follow from (\ref{ContinuousPolicies}). The inequality (\ref{inequality}) follows from the fact that the value function will be monotonically increasing in $M$: this is because the sequence of information structures is a minimizer-garbling sequence\footnote{It is achieving this inequality which requires the sequence of information structures to be a minimizer-garbling sequence, rather than an arbitrary weakly converging sequence of information structures.}. The second equality follows from the weak convergence of $\{\mu\}_m$ to $\mu$.
Since the above holds for all $M \in \mathbb{Z}_{\geq 0}$, it will also hold in the limit of $M$, and thus:
\begin{align*}
    &\limsup_{m \rightarrow \infty} \inf_{\gamma^1 \in \Gamma^1}\sup_{\gamma^2 \in \Gamma^2}\int c(x, \gamma^1(y^1), \gamma^{2}(y^{2}))\mu_m(dx,dy^1,dy^{2}) \nonumber \\
    &\leq \inf_{\gamma^1 \in \Gamma^1}\sup_{\gamma^2 \in \Gamma^2}\int c(x, \gamma^1(y^1), \gamma^{2}(y^{2}))\mu(dx,dy^1,dy^{2}). \nonumber
\end{align*}

Thus, the value function is upper semicontinuous. A similar proof applies for maximizer-garbling sequences.\end{proof}

We now consider setwise convergence of information structures; the proof follows closely that of Theorem \ref{zerosumweak}.

\begin{theorem}\label{zerosumsetwise}
Consider a fixed prior and let Assumptions A1, A3, A4 and A5 hold. Let $\mu_m$ be a sequence of information structures converging setwise to information structure $\mu$. If the sequence is a maximizer-garbling, then the value function is lower semicontinuous. If the sequence is a minimizer-garbling, then the value function is upper semicontinuous. 
\end{theorem}

\begin{proof}
Consider a maximizer-garbling sequence. We have that:
\begin{align}
 & \liminf_{m \rightarrow \infty} \inf_{\gamma^1 \in \Gamma^1} \sup_{\gamma^2 \in \Gamma^2}\int c(x, \gamma^1(y^1), \gamma^2(y^2))\mu_m(dx,dy^1,dy^{2}) \nonumber \\
 & \geq \sup_{\gamma^2 \in \Gamma^2} \liminf_{m \rightarrow \infty}\inf_{\gamma^1 \in \Gamma^1}\int c(x, \gamma^1(y^1), \gamma^{2}(y^{2}))\mu_m (dx,dy^1,dy^{2}). \nonumber
\end{align}
 The inequality above follows from interchanging the limit inferior and the supremum. Following a similar reasoning as the proof of Theorem \ref{zerosumweak}, we have the following for any fixed policy $\gamma^2 \in \Gamma^2$:
\begin{align}
    &\liminf_{m \rightarrow \infty} \inf_{\gamma^1 \in \Gamma^1}\int c(x, \gamma^1(y^1), \gamma^{2}(y^{2}))\mu_m(dx,dy^1,dy^{2}) \nonumber \\
 & \quad   \geq \inf_{\gamma^1 \in \Gamma^1}\int c(x, \gamma^1(y^1), \gamma^{2}(y^{2}))\mu(dx,dy^1,dy^{2}). \nonumber
\end{align}
This gives us
\begin{align}
 & \liminf_{m \rightarrow \infty} \inf_{\gamma^1 \in \Gamma^1} \sup_{\gamma^2 \in \Gamma^2}\int c(x, \gamma^1(y^1), \gamma^2(y^2))\mu_m(dx,dy^1,dy^{2}) \nonumber \\
& \quad \geq \sup_{\gamma^2 \in \Gamma^2} \inf_{\gamma^1 \in \Gamma^1} \int c(x, \gamma^1(y^1), \gamma^{2}(y^{2}))\mu (dx,dy^1,dy^{2}). \nonumber
\end{align}
Thus, the value function is lower semicontinuous. A similar proof can show the result for minimizer-garbling sequences.  \end{proof}

We can also show that when the player channels are fixed, independent and total-variation continuous, the value function is continuous under weak convergence of the priors. This generalizes a single-player result from \cite{KYSICONPrior}. 

\begin{theorem}\label{ZSprior}
Under Assumptions A2, A4, and A6, the value function is continuous under weak convergence of the priors. 
\end{theorem}


\begin{proof}
Let $\zeta_m$ be a sequence of priors on $\mathbb{X}$ converging weakly to $\zeta$. Following the same reasoning as in Theorem \ref{continuity}, if for some $j$, $J^*(g, Q^1\zeta Q^2) - J^*(g, Q^1\zeta_j Q^2) \geq 0$, then we have:
\begin{align}
    & J^*(g, Q^1\zeta Q^2) - J^*(g, Q^1\zeta_j Q^2) \nonumber \\ 
    &\leq \int c(x, {\gamma}^{1,*}_{j}(y^1), \gamma^{2,*}(y^2))Q^1(dy^1|x)Q^2(dy^2|x)\zeta(dx) \nonumber \\
    & \qquad \qquad - \int c(x, {\gamma}^{1,*}_{j}(y^1), {\gamma}^{2,*}(y^2))Q^1(dy^1|x)Q^2(dy^2|x)\zeta_j(dx) \nonumber \\
     &= \int(\zeta - \zeta_j)(dx)\int c(x, {\gamma}^{1,*}_{j}(y^1), \gamma^{2,*}(y^2))Q^1(dy^1|x)Q^2(dy^2|x). \nonumber
\end{align}
Considering the case $J^*(g, Q^1\zeta Q^2) - J^*(g, Q^1\zeta_j Q^2) \leq 0$ also as above we get 
\begin{align}
    |J^*(g,&Q^1\zeta Q^2) - J^*(g, Q^1\zeta_m Q^2)| \nonumber \\ \leq &\max\{|\int(\zeta - \zeta_m)(dx)\int c(x, {\gamma}^{1,*}_{m}(y^1), \gamma^{2,*}(y^2))Q^1(dy^1|x)Q^2(dy^2|x)|,  \nonumber \\  & \quad \quad \quad \quad |\int(\zeta - \zeta_m)(dx)\int c(x, {\gamma}^{1,*}(y^1), \gamma^{2,*}_{m}(y^2))Q^1(dy^1|x)Q^2(dy^2|x)|\} \label{maxterms}.
\end{align}
Then, following the proof of \cite[Theorem 2.5]{KYSICONPrior}, since the channels are continuous in total variation, 
\[\int c(x, {\gamma}^{1,*}_{m}(y^1), \gamma^{2,*}(y^2))Q^1(dy^1|x)Q^2(dy^2|x),\] and 
\[\int c(x, {\gamma}^{1,*} y^1), \gamma^{2,*}_m(y^2))Q^1(dy^1|x)Q^2(dy^2|x),\] are equicontinuous families of functions. Using this fact and that $\zeta_m$ converges weakly to $\zeta$, by \cite[Corollary 11.3.4]{dud02} it follows that both terms in (\ref{maxterms}) converge to zero, and so $|J^*(g, Q^1\zeta Q^2) - J^*(g, Q^1\zeta Q^2)| \rightarrow 0$. 
\end{proof}


\section{Stochastic Teams}\label{TeamProblems}
Now we evaluate the regularity properties of the value function for team problems in information structures under total variation, weak, and setwise convergence. 
\begin{theorem}\label{teamtheorem}
Consider an $n$ decision-maker team and let Assumptions A1, A3, A4 and A5 hold for a fixed team problem. Then the value function is continuous under total variation convergence of information structures. 
\end{theorem}
\begin{proof}

An optimal solution exists for any given information structure under the assumptions \cite[Section 5]{YukselWitsenStandardArXiv}. Let $\mu$ and $\nu$ be two information structures, and let $\gamma_\mu^{i,*}$ and $\gamma_{\nu}^{i,*}$ denote Player $i$'s individual strategy as part of the team-optimal strategies $\bar{\gamma}^*_\mu$, $\bar{\gamma}^*_\nu$ under each respective information structure. Let the cost function be bounded in absolute value by $M$. 

Without loss of generality, assume $J^*(c, \mu) - J^*(c,\nu) \geq 0$. Then
\begin{align*}
    &J^*(c, \mu) - J^*(c, \nu)  = \int c(x, \gamma^{1,*}_{\mu}(y^1), \dots, \gamma^{n,*}_{\mu}(y^n))\mu(dx, dy^1, \dots, dy^n) \\ & \quad \quad \quad \quad \quad \quad \quad \quad \quad \quad \quad \quad - \int c(x, \gamma^{1,*}_{\nu}(y^1), \dots, \gamma^{n,*}_{\nu}(y^n))\nu(dx, dy^1, \dots, dy^n) \\
    & \leq \int c(x, \gamma^{1,*}_{\nu}(y^1), \dots, \gamma^{n,*}_{\nu}(y^n))\mu(dx, dy^1, \dots, dy^n) \\ & \quad \quad \quad \quad - \int c(x, \gamma^{1,*}_{\nu}(y^1), \dots, \gamma^{n,*}_{\nu}(y^n))\nu(dx, dy^1, \dots, dy^n) \leq M\|\mu - \nu \|_{TV}.
\end{align*}
Where the first inequality holds by perturbing the strategy under $\mu$ to be $\bar{\gamma}^*_\nu$, rather than the team-optimal policy $\bar{\gamma}^*_\mu$, and the second inequality holds by the definition of total variation distance. The result follows. 
\end{proof}

\begin{theorem}
If player action spaces are compact and Assumptions A2, A3, A4 and A6 hold, the value function is continuous under weak convergence of the priors. 
\end{theorem}

\begin{proof}
Let $\zeta_m$ be a sequence of priors on $\mathbb{X}$ converging weakly to $\zeta$. Without loss of generality we assume that $J^*(g, Q^1\dots Q^n\zeta) - J^*(g, Q^1\dots Q^n\zeta_m) \geq 0$ for some game $g$. 
Following the same procedure as in Theorem \ref{teamtheorem}, we get
\begin{align}
    &J^*(g, Q^1\dots Q^n\zeta) - J^*(g, Q^1\dots Q^n\zeta_m)& \nonumber \\ &\quad \leq \int(\zeta - \zeta_m)(dx)\int c(x, {\gamma}^{1,*}_{m}(y^1), \dots, \gamma^{n,*}_{m}(y^n))Q^1(dy^1|x)\dots Q^n(dy^n|x). \nonumber
\end{align}
Continuing as in the proof of Theorem \ref{ZSprior}, using equicontinuity, the result follows. 
\end{proof}

\begin{theorem}
Assume Assumptions A1, A3, A4, A5, and A7 hold. Let $\mu_m$ be a sequence of information structures converging weakly to an information structure $\mu$. If the prior is not fixed, then we also impose continuity in $x$ (under A2). Under these conditions the value function is upper semicontinuous in $\mu$ under weak convergence.
\end{theorem}
\begin{proof}
Let $\bar{\Gamma} = \Gamma^1 \times \dots \times \Gamma^n$. Let $\mu_{{\mathbb{Y}}^j}$ be the marginal of $\mu$ on its $(j + 1)$th component. Let $\gamma^j$ be an arbitrary policy for player $j$. The action spaces for players are convex by Assumption A4. Let $M = \|c\|_{\infty}$. Then, by Lusin's theorem, for any $\epsilon >0$, there exists a continuous function $f^j: \mathbb{Y}^j \rightarrow \mathbb{U}^j$ such that $\mu_{{\mathbb{Y}}^j}(\{y^j : f^j(y^j) \neq \gamma^j(y^j)\}) < \epsilon.$

Letting $B^j$ denote the set $B^j = \{y^j : f^j(y^j) \neq \gamma^j(y^j)\}$, and proceeding with this for $j = 1, \dots, n$, we have:
\begin{align}
&\int c(x, f^1(y^1), \dots, f^n(y^n))\mu(dx, dy^1, \dots, dy^n)  \nonumber \\
& \quad \quad < \int c(x, \gamma^1(y^1), \dots, \gamma^n(y^n))\mu(dx, dy^1, \dots, dy^n) + nM\epsilon. \nonumber   
\end{align}

Denote by $\Gamma^j_C$ the policy space for player $j$ in which the player's policy is continuous from $\mathbb{Y}^j$ to $\mathbb{U}^j$. By the above, it follows that the value for a game will be the same if each player uses $\Gamma$ or $\Gamma_C$. 
Thus, applying the fact that $\mu_m$ converges weakly to $\mu$: 
\begin{align*}
    &\limsup_{m \rightarrow \infty} \inf_{\bar{\gamma} \in \bar{\Gamma}} \int c(x, \gamma^1(y^1), \dots, \gamma^n(y^n))\mu_m(dx, dy^1, \dots, dy^n) \\ 
    &= \limsup_{m \rightarrow \infty} \inf_{\bar{\gamma}_C \in \bar{\Gamma}_C} \int c(x, \gamma^1_C(y^1), \dots, \gamma^n_C(y^n)) \mu_m(dx, dy^1, \dots, dy^n) \\
    &\leq \inf_{\bar{\gamma}_C \in \bar{\Gamma}_C} \limsup_{m \rightarrow \infty} \int c(x, \gamma^1_C(y^1), \dots, \gamma^n_C(y^n)) \mu_m(dx, dy^1, \dots, dy^n) \\
    &= \inf_{\bar{\gamma}_C \in \bar{\Gamma}_C} \int c(x, \gamma^1_C(y^1), \dots, \gamma^n_C(y^n))\mu(dx, dy^1, \dots, dy^n) \\
    &= \inf_{\bar{\gamma} \in \bar{\Gamma}} \int c(x, \gamma^1(y^1), \dots, \gamma^n(y^n))\mu(dx, dy^1, \dots, dy^n).
\end{align*}

We note that if the prior is fixed, in the inequality above we do not need continuity of $c$ in $x$ since weak convergence of a product measure with a fixed marginal is equivalent to $w$-$s$ convergence of the joint measure, see Section \ref{convergenceIS}. If the prior is not fixed, then we impose continuity in $x$ also under A2 and the inequality holds.
\end{proof}

Following a similar argument, without Lusin's theorem, we have the following.
\begin{theorem}
Assume Assumptions A1, A3, A4 and A5 hold. Let $\mu_m$ be a sequence of information structures converging setwise to an information structure $\mu$. Then the value function is upper semicontinuous in $\mu$ under setwise convergence. 
\end{theorem}

\begin{proof}
By interchanging the limit superior and the infimum, and then applying the fact that $\{\mu_m\} \rightarrow \mu$ setwise, we get: 
\begin{align*}
 &\limsup_{m \rightarrow \infty}\inf_{\bar{\gamma} \in \bar{\Gamma}}\int c(x, \gamma^1(y^1), \dots, \gamma^n(y^n))\mu_m(dx, dy^1, \dots, dy^n) \nonumber \\
 &\leq \inf_{\bar{\gamma} \in \bar{\Gamma}}\limsup_{m \rightarrow \infty} \int c(x, \gamma^1(y^1), \dots, \gamma^n(y^n))\nonumber \mu_m(dx, dy^1, \dots, dy^n) \nonumber \\
 & = \inf_{\bar{\gamma}\in \bar{\Gamma}}\int c(x, \gamma^1(y^1), \dots, \gamma^n(y^n))\mu(dx, dy^1, \dots, dy^n)\nonumber.
\end{align*}
\end{proof}

\section{General Games}

While we have shown that zero-sum games and team problems both exhibit continuity under total variation convergence of information structures for games with measurable and bounded cost functions, here we present a counterexample that reveals this is not true for general non-zero-sum games.

\begin{example}
Consider a two-player game with state space $\mathbb{X}$ = $[-1,1]$ endowed with the continuous uniform distribution as $\zeta$. Let $\mathbb{Y}^1 = \{1,2,3\}$ and $\mathbb{Y}^2 = \{0\}$, and $\mathbb{U}^1 = \mathbb{U}^2 = [-1,1]$. For $m \in \mathbb{Z}_{\geq 1}$, define information structure $\mu_m$ by channels $Q^1_m$ and $Q^2_m$ for players 1 and 2 respectively, where $Q^1_m$ is a quantizer with three bins:

\begin{equation}
y_m^1 = \begin{cases}
    1, \quad x \in [-1, -\frac{1}{2} - \frac{1}{8m}) \\ 2, \quad x \in [-\frac{1}{2} -\frac{1}{8m}, \frac{1}{2} + \frac{1}{4m}] \\ 3, \quad x \in (\frac{1}{2} + \frac{1}{4m}, 1] \nonumber
    \end{cases}    
\end{equation}

and $Q^2_m$ returns $y^2 = 0$ for all $x \in \mathbb{X}$. 

Following \cite[Theorem 5.7]{YukselOptimizationofChannels}, since our information structure is defined by quantizers that converge setwise at input $\zeta$, $\{\mu_m\} \rightarrow \mu$ in total variation, where $\mu$ is defined by the quantizer for Player 1 which sorts $x$ into the bins $[-1, -1/2)$, $[-1/2, 1/2]$, and $(1/2, 1]$, and player 2 has the same channel that always returns $y^2 = 0$. 

Now we define the following cost functions for the players:
\begin{equation}
        c^1(x,u^1,u^2) = (x-u^1)^2 - (u^2)^2, \quad \quad c^2(x,u^1,u^2) = \begin{cases} (u^2)^2, \quad u^1 = 0 \\  (u^2-1)^2, \quad u^1 \neq 0
    \end{cases}  \nonumber
\end{equation}
For each $m$, Player 1's optimal strategy is to minimize $(x - u^1)^2$, and thus his optimal policy is to play $u^1 = E_{\mu_n}[x|y^1]$. Due to the asymmetry of Player 1's information channel, $u^1 \neq 0 \: \forall m$. Thus, Player 2's optimal strategy is to play $u^2 = 1$, and her expected cost is $0$. 

However, under $\mu$, Player 1 will play $0$ with probability $1/2$. Thus, Player 2's optimal strategy is to play $u^2 = 1/2$, and her expected cost is $1/4$. Thus, the value function for Player 2 is not continuous under total variation convergence of the information structure. 
\end{example}

We conclude that general non-zero-sum games do not necessarily exhibit continuity under total variation, weak, or setwise convergence of information structures for games with measurable and bounded cost functions (as total variation is a stronger notion of convergence than both weak and setwise convergence).


\section{Conclusion}
We presented continuity properties of value functions and equilibrium solutions in zero-sum, team, and general game problems with respect to information structures.
It was shown that the value function for both zero-sum games and team problems is continuous under total variation convergence of information structures. In both cases, the change in expected value when switching between two information structures is bounded above by the product of the total variation distance and the $L^{\infty}$ norm of the cost function. For zero-sum games, the value function is upper semicontinuous for minimizer-garbling sequences of information structures, and lower semicontinuous for sequences of maximizer-garbling information structures under weak (if the cost function is continuous and bounded and the action spaces are convex) or setwise (if the cost function is measurable and bounded) convergence. For team problems, the value function is upper semicontinuous under setwise convergence for measurable and bounded cost functions, and upper semicontinuous under weak convergence for bounded and continuous cost functions when the players' action spaces are convex. A counterexample revealed players in general non-zero-sum games may not have value functions that are continuous under total variation convergence of information structures, even when the cost functions are bounded. 

While we studied static games and teams in the paper, since it is known that under absolute continuity conditions there is an isomorphism relationship between equilibrium solutions to dynamic teams/games and their static reductions (which turns out to be policy independent) \cite{sanjari2021gameoptimality}, the results also apply to such dynamic games with absolutely continuous information structures.

\bibliographystyle{plain}
\bibliography{IanBib,SerdarBibliography}

\end{document}